\documentclass[11pt]{amsart}
\usepackage{amssymb,amsmath,amsfonts,amsthm,amscd,enumerate}

\usepackage{graphics}
\usepackage[colorlinks=true]{hyperref}
\usepackage{amssymb,amsmath,amsfonts,amsthm,amscd,mathrsfs}
\usepackage{fullpage}

\input xy
\xyoption{all}

\newcommand{\RR}{\mathbb R}
\newcommand{\R}{\mathbb R}
\newcommand{\N}{\mathbb N}
\newcommand{\maA}{\mathcal{A}}
\newcommand{\maH}{\mathcal{H}}

\newcommand{\maB}{\mathcal{B}}

\newcommand{\maD}{\mathcal{D}}

\newcommand{\ZZ}{\mathbb{Z}}
\newcommand{\Sch}{\mathscr{S}}
\newcommand{\NN}{\mathbb{N}}

\newcommand{\maM}{\mathcal{M}}
\newcommand{\maE}{\mathcal{E}}

\newcommand{\maG}{\mathcal{G}}

\newcommand{\pdo}[1]{\Psi^{\infty}(#1)}
\newcommand{\smooth}{\mathcal{C}^{\infty}}
\newcommand{\smoothing}{\Psi^{\!-\!\infty}}

\newcommand{\Op}{\operatorname}
\newcommand{\ep}{\varepsilon}

\newcommand{\ip}[1]{\langle #1 \rangle}

\theoremstyle{definition}
\newtheorem{definition}{Definition}[section]
\newtheorem{eg}[definition]{\bf Example}

\theoremstyle{plain}
\newtheorem{theo}[definition]{Theorem}
\newtheorem{prop}[definition]{Proposition}
\newtheorem{lem}[definition]{Lemma}

\theoremstyle{remark}
\newtheorem{remark}[definition]{Remark}

\newcommand{\WF}{\mathrm{WF}}

\begin{document}

 \title[]{Singularity structures for noncommutative spaces} 

\author{Shantanu Dave}
\address{University of Vienna, Faculty of Mathematics, Nordbergstrasse 15, 1090 Vienna, Austria}
\email{shantanu.dave@univie.ac.at}
\thanks{Supported by FWF grants P20525 and P24420 of the Austrian Science Fund.}
\author{Michael Kunzinger}
\address{University of Vienna, Faculty of Mathematics, Nordbergstrasse 15, 1090 Vienna, Austria}
\email{michael.kunzinger@univie.ac.at}
\thanks{Supported by FWF grants Y237-N13 and P20525 of the Austrian Science Fund.}

\begin{abstract}We introduce a (bi)category $\mathfrak{Sing}$  whose objects can be functorially assigned spaces of distributions and generalized functions. In addition, these spaces of distributions and generalized functions possess intrinsic notions of regularity and singularity  analogous to  usual Schwartz distributions on manifolds. The objects in this category can be obtained from smooth manifolds, noncommutative spaces, or Lie groupoids. An application of these structures relates the propagation of singularities on a groupoid with that  on the base manifold.

\end{abstract}

\maketitle

\section{Introduction}
A precise description of the propagation of singularities under the solution operators of hyperbolic partial differential equations depends on the formulation of the notions of singular support and wavefront sets for distributions.  The propagation of singularities can then be viewed as a link between the wave and corpuscular theory of light.  There are many applications in geometry and analysis of this phenomenon, particularly to the spectral theory of  elliptic pseudo-differential operators. 

These phenomena are even more remarkable in the case of manifolds with  boundaries or corners (see  \cite{Melrose,MVW,Vasy}).  In principle, it is evident that propagation of singularities occurs as the solution operators to hyperbolic partial differential equations are constructed via a parametrix construction related to some calculus of pseudo-differential operators, whereby the dynamics is transferred to the space of principal symbols. In many situations involving manifolds with corners, the relevant calculus can be constructed based on differentiable groupoids  \cite{BM,NWX,LN}. 
In this paper we present  a category-theoretic view of distribution theory, including singularities and propagation of singularities.  We also consider morphisms that relate these propagations. As an application, we relate the longitudinal 
propagation of singularities on certain groupoids to the transverse propagation of singularities on the base manifold with respect to vector representation of longitudinal operators.

First we consider a category $\mathfrak{Sing}$ whose objects,  referred to as {\em singularity structures}, are triples $(A,X,Y)$ where $X$ and $Y$ are suitable Fr\'echet modules over a filtered algebra $A$.  In this category, an instance of a closed manifold $M$ is represented by $A=\Psi^{\infty}(M)$, the algebra of (classical) pseudo-differential operators, $X=\smoothing(M)$, the ideal of regularizing operators and $Y=\smooth(M)$. More generally, objects in $\mathfrak{Sing}$ are provided by noncommutative spaces  represented by (regular) spectral triples in the sense of Connes, geometric Hilbert spaces, and groupoids. 

The association of distributions to a singularity structure follows from a simple observation. In the classical case of
a closed manifold $M$, represented as above, the space of distributions, that is the dual space to the space of densities $\maD'(M)=|\Lambda|(M)'$, can be realized exactly as a left $A$ module map between $X$ and $Y$. 
As a direct generalization, we consider abstract generalized functions and distributions  on any singularity structure $(A,X,Y)$ and obtain in a natural way the notion of regularity and singularity of abstract distributions and generalized functions associated to a general triple $(A,X,Y)$. Thus
 Sobolev type regularity (by choice of a hypo-elliptic, or a self-adjoint positive  operator and a scale), and constructions of singular supports and wavefront sets, are built in by definition. Regularity in this sense can be described entirely by the graded Fr\'echet space structures on $X$ and $Y$, however the  singularities depend crucially on the choice of the algebra $A$. 
 
 In the classical analysis of linear PDEs the main benefit of singularity and regularity analysis of distributions lies in the use of their compatibility with respect to pull-backs and push-forwards of distributions under appropriate smooth maps (see \cite{GS} Chapter 6 for some interesting applications). Thus a class of morphisms is introduced so that
   the microlocal information defined for an $(A,X,Y)$-generalized function behaves functorially with respect to them. 
The obvious class of morphisms of these singularity structures is very restrictive, hence it is necessary to consider a larger class of morphisms essentially in the spirit of Morita equivalence. It is expected that suitably regular KK-cycles,  possibly with additional structure, will provide morphisms of the singularity structures, especially the ones arising from spectral triples.  Here we shall focus on morphisms defined out of vector representations on Lie-groupoids. This is the most relevant case to understand remarkable propagation results in, e.g., \cite{Hormander, Melrose, Vasy}. However, the assumptions in this article are much stronger, which simplifies the exposition. We expect similar results to hold for differential groupoids associated to manifolds with corners. 

The interpretation of the propagation of singularity phenomenon in this new context is akin to an action by  $\RR$. This can be visualized as an abstract Egorov theorem, which in the classical context provides an $\RR$-action on the singularity structure  associated to a manifold as above.  Thus, a comparison of the dynamics associated to various singularity structures with $\RR$-actions is now possible via considering equivariant morphisms between them.
Finally, as already mentioned, an application of the category  $\mathfrak{Sing}$ is to show that the anchor map  on the Lie algebroid $A(\maG)$ relates the longitudinal propagation of singularities on a groupoid with the transverse propagation of singularities on the base manifold. 
 
\section{Algebraic basics}
This section introduces the notation used in this article. We shall also collect some examples for easy reference.

We shall consider an algebra $A$, a right $A$-module $X$  and a left  $A$-module $Y$ given to us. The space of (set-theoretic) maps  from $X$ to $Y$ will be denoted by $\maM(X,Y)$  and is given a natural  $A$-bimodule structure. For any $\phi:X\rightarrow Y$ and $a\in A$, this structure is defined by
\[(a\phi)(x):=a\phi(x)\quad (\phi a)(x):=\phi(xa).\]
For our purposes both the modules $X$ and $Y$  shall be  Fr\'echet spaces  provided with a grading. Recall that a  grading on a Fr\'echet space $X$  is a  sequence of  seminorms $\|\,\|_n$ that  is increasing   $\|\,\|_1\leq\|\,\|_2\leq\ldots $ and such that it generates the locally convex topology  on $X$ (cf.\ 
\cite{Hamilton}, Def.\ 1.1.1).

Additionally we shall assume that the algebra $A$ is filtered. This  entails that there exist subspaces $A_i\subset A$  for each $i\in\ZZ$ such that $ A_i\subseteq A_{i+1}$  and $A_i\cdot A_j\subseteq A_{i+j}$ for all $i$, $j$.
\begin{definition}\label{Fmodule}
A (right) module $X$ over a filtered algebra $A$ is called a  Fr\'echet module if $X$ is a graded Fr\'echet space and for
each $P\in A_j$ there exists a constant $C>0$ such that for all $n\in \N$:
\[\|xP\|_n\leq C\|x\|_{n+j}~~~\textrm{for all}~~x\in X.\]
Analogous notions for left and bi-modules over $A$ will be understood.
\end{definition}
 The following two  examples provide a strong motivation for the construction developed in
 this paper.

 \begin{eg}\label{pseudo_closed}
  Let $A=\Psi^{\infty}(M)$  be the algebra of  (classical) pseudo-differential operators on a closed manifold $M$ with the usual filtration given by the degree of the operator.
  
   We shall grade  $\smooth(M)$  with Sobolev norms. To this end, let $\Delta$ denote  the Laplace operator on $M$ associated to some Riemannian metric on $M$ and set
\[\|f\|_n:=\|(1+\Delta)^{\frac{n}{2}}f\|_{L^2(M)}.\]
 With the above grading, $\smooth(M)$  is a graded Fr\'echet algebra  in the sense of Definition  \ref{FA} below. 

For any operator $T:L^2(M)\rightarrow L^2(M)$ let $\|D\|_{\operatorname{HS}}$  denote its  Hilbert-Schmidt norm.
Given integers $p$, $q$ we set  
$\|T\|_{p,q}:=\|(1+\Delta)^{\frac{p}{2}}T(1+\Delta)^{\frac{q}{2}}\|_{\operatorname{HS}}$.
 
 Throughout this article, the  grading on $\smoothing(M)$ will be defined  by:
\[
\|T\|_n:=\sum_{\substack{q+p\le n \\ q,p\geq -n}	
}
\|T\|_{p,q}.
\]
Since $\smoothing(M)$ is a two-sided ideal in $A=\Psi^{\infty}(M)$, we shall take this module  structure and  consider the natural (left) action of $A$  on $\smooth(M)$.
Both $\smooth(M)$ and $\smoothing(M)$ are then  Fr\'echet modules over $A$ 

We note that by the Schwartz kernel theorem, given a choice of non-vanishing density on $M$ one can identify $\smoothing(M)$ with $\smooth(M\times M)$.  The above grading on $\smoothing(M)$ is then equivalent to the tensor-product grading with the chosen Sobolev grading on $\smooth(M)$.

 \end{eg}
 
 Next we consider a more general example coming from a regular spectral triple in the sense of Connes \cite{NCG}.
 
 Let  $D$ be an unbounded self-adjoint  operator on a Hilbert space  $H$ with compact resolvent. Let $\Delta:=D^2$.
  We can define a  Sobolev space $\maH^s$ for $s\in \RR$  as usual by means of completion with respect to the seminorms
  \[\|x\|_s^2:=\|x\|^2+\|\Delta^{\frac{s}{2}} x\|^2.\]

   The core of $D$, given by $\maH^{\infty}:=\bigcap_s\maH^s$ is then  naturally a  Fr\'echet space.
\begin{eg}   
   An algebra $\maA$ of  pseudodifferential operators such that $\maH^{\infty}$ is a Fr\'echet   $\maA$ module can be obtained from an involutive algebra $A$ of bounded operators on $H$ which satisfies a certain regularity condition, namely that $(A,H,D)$ form a regular spectral triple. We briefly recall the  notion of a regular spectral triple  and the associated algebra of pseudo-differential operators. To describe this algebra of pseudo-differential operators we follow the presentation in \cite{Higson}.
 
     A spectral triple   consists of an involutive algebra of operators $A$ on a Hilbert space $H$  and a specified unbounded densely defined self-adjoint operator $D$ such that $D$ has compact resolvent, the elements in $A$ preserve $\operatorname{Domain}(D)$, and the  commutators $[D,a]$ extend to  bounded operators on $H$ for all $a\in A$.
     
    Let $\delta$ be the unbounded  derivation defined on the Banach algebra $\maB(H)$ of bounded operators on $H$ by
    \[\delta(T):=[|D|,T]\qquad T~\textrm{preserves}~\operatorname{Domain}(|D|).\]
  A spectral triple $(A,H,D)$ is called regular if for any $a\in A$  and any positive integer $n$  both $a$ and $[D,a]$ are in $\operatorname{Domain}(\delta^n)$.
  
  We begin by defining an algebra of differential operators $\maD:=\bigcup_{k=0}^{\infty}\maD_k(A)$ for a regular spectral triple. To this end, first set  the algebra of order zero differential operators to be $\maD_0:=A+[D,A]$, and then inductively define a filtered algebra $\maD:=\bigcup_k\maD_k$ by
  \begin{eqnarray}\label{gen_filtration}
 \maD_k:=\sum \maD_i\cdot \maD_{k-i}+[\Delta,\maD_{k-1}]+\maD_0[\Delta,\maD_{k-1}]+\maD_{k-1}.
 \end{eqnarray}
   
   \begin{definition}\label{pseudo_def}
We shall call an operator $T$ on $\maH^{\infty}$  a basic  pseudo-differential operator of order $k$ if for any $l\in\ZZ$  there exist $m\in \ZZ$, $R$ and $X$ so that
\[T=X\Delta^{\frac{m}{2}}+R,\]
 where $X\in\maD$, $order(X)\leq k-m$, and the operator $R:\maH^s\rightarrow \maH^{s-l}$ is bounded for all $s\in\RR$.
 
 More generally, a pseudo-differential operator is a finite linear combination of basic pseudo-differential operators. The algebra of all pseudo-differential operators shall be denoted by $\maA:=\maA_{(A,H,D)}$.
 \end{definition}
 
 The following result is a source of examples of Fr\'echet modules. 
 \begin{prop}
  Let $(A,H,D)$ be a regular spectral triple such that $A$ maps $\maH^{\infty}$ to itself. Then $\maH^{\infty}$ is a Fr\'echet module over $\maA$.
 \end{prop}
 \begin{proof} An equivalent statement is proved in \cite{CM} Appendix B. We shall follow the presentation in \cite{Higson} and give a proof here for completeness. 
 
 Since the spectral triple is regular we can define an algebra $\Psi$  generated by all elements of the form $\delta^n(b)$, where $b\in A$ or $b\in [A,D]$.
  By definition, all elements of $\Psi$ extend to bounded operators on $H$ and $\delta(\Psi)\subseteq \Psi$. Next, we define a new algebra $\maB$ of elements of the form $b|D|^k$, where $b\in \Psi$ and $k\geq 0$. Note that
  \[b_1|D|^{k_1}b_2|D|^{k_2}=\sum_j
  \begin{pmatrix}
	k_1 \\ b_1
\end{pmatrix}
  \delta^j(b_2)|D|^{k_1+k_2-j}, \] 
which proves that $\maB$ is a filtered algebra with $\maB_k$ generated by elements of the form $b|D|^j$ with $j\leq k$

It is clear that $\maH^{\infty}$ is a graded Fr\'echet module over $\maB$. Simply note that for any $v\in\maH^{\infty}$ and $b|D|^j\in \maB_k$,  for $j\leq k$ we have
\[\|b|D|^lv\|\leq\|b\|\|D^lv\|\leq \|b\|(\|D^kv\|+\|v\|),\]
where the last inequality follows from spectral theory. 

We now claim that the algebra $\maA$ is a filtered subalgebra of $\maB$. To this end it suffices to show that the algebra of differential operators $\maD$ is a filtered subalgebra of $\maB$. We first note that $\maD_0\subseteq \maB_0$. Based on this we may use induction to see that if $\maD_{k-1}\subseteq \maB_{k-1}$, the calculation
\[[b|D|^{k-1},|D|^2]=2\delta(b)|D|^k+\delta^2(b)|D|^{k-1}\]
  shows that $\maD_k\subseteq \maB_k$.
\end{proof}

  Since the operator $D$ is assumed to have compact resolvents, it follows that $\maH^{\infty}$ is a Montel space.  The consequence of interest is that $\maH^{\infty}$ is reflexive.   Thus we consider the  space of linear maps from the  dual space $\maH^{-\infty}=\bigcup_s H^s$ to $\maH^{\infty}$ as smoothing operators. In addition, when the spectral triple is $p$-summable then  $\maH^{\infty}$ is nuclear. This justifies that we define\footnote{We denote by $\hat{\otimes}$ the projective tensor product.}
  \[\smoothing(A,H,D):=\left(\maH^{-\infty}\right)^* \hat{\otimes}H^{\infty}=\maH^{\infty}\hat{\otimes}\maH^{\infty}.\]
Here we take the tensor product graded Fr\'echet structure on $\smoothing(A,H,D)$. 
   \begin{prop}If $(A,H,D)$ is a $p$-summable regular spectral triple and $A$ maps $\maH^{\infty}$ to itself, then $\smoothing(A,H,D)$ is a graded Fr\'echet module.
   \end{prop}
   The essential idea of the proof is the same as in the special case in Example \ref{orderm} below.
  \end{eg} 



\section{Abstract Regularity}\label{PT}

We first  consider regularity of maps between  Fr\'echet spaces.  We refer to \cite {Hamilton} for further  study. 
\begin{definition} 
Let $X$ and $Y$  be graded Fr\'echet spaces.  We denote by $\|.\|_n$ and $\|.\|'_n$ the $n$-th graded norm on $X$  and $Y$, respectively. We say that a Fr\'echet  smooth map $\phi:X\rightarrow Y$ is polynomially  tame if there exist $b,k\in\NN$  and some $r\in\ZZ$ such that
\begin{eqnarray}\label{tamer}
\|\phi(x)\|'_n\leq C_n\|x\|^k_{n+r}\quad \textrm {for all}\, n\geq b+|r|.
\end{eqnarray}
Here $C_n>0$ is a constant that depends only on $n$. If $k=1$ we call $\phi$ linearly tame.
The number $r$  is called the degree of tameness and the set of all maps  of tameness degree $r$  is denoted by $\Op{PT}^r(X,Y)$. 
Then clearly $\Op{PT}^r(X,Y) \subseteq \Op{PT}^s(X,Y)$ for $r\le s$.
We set $\Op{PT}(X,Y):=\bigcup_r\Op{PT}^r(X,Y)$. 

\begin{remark}\label{lintameremark} Note that a nontrivial linear map $\phi$ is polynomially tame if and only it is linearly tame.
Indeed, if $k$ was greater $1$ in (\ref{tamer}) then replacing $x$ by $\lambda x$ ($\lambda \in\R$) would 
entail $\phi=0$. 
\end{remark}

A polynomially tame map is called regular if it is tame of all orders. We denote by $\Op{Reg}(X,Y)$ the space of all regular maps between $X$ and $Y$, i.e.,
 \[\Op{Reg}(X,Y)= \bigcap_r\Op{PT}^r(X,Y).\]
\end{definition}
 We will also require the following tameness  property for  an associative multiplication on a Fr\'echet  space.
\begin{definition}\label{FA}
We say that $X$ is a Fr\'echet  algebra if it is a Fr\'echet space with an associative product and the multiplication is jointly continuous. A graded Fr\'echet algebra is a Fr\'echet algebra that is a graded Fr\'echet space and such that the   multiplication satisfies the following tameness  condition: there exist $b,r_1,r_2\in\NN$ such that
\[\|x\cdot y\|_n\leq C_n\|x\|_{n+r_1}\|y\|_{n+r_2}\quad \forall \,\,n\geq b.\]
\end{definition}
With the above definition the following lemma is self-evident.
\begin{lem}\label{composition}
Let $Y$ be a graded Fr\'echet algebra and let $X$ be any graded Fr\'echet space. Then the space 
$\Op{PT}(X,Y)$ of polynomially tame maps from $X$ to $Y$ is an algebra under pointwise operations. The space of regular maps $\Op{Reg}(X,Y)$ is an ideal in $\Op{PT}(X,Y)$. Moreover, given graded Fr\'echet spaces $X,Y,Z$ and $\phi\in PT(X,Y)$ and $\psi\in \Op{Reg}(Y,Z)$, then $\psi\circ\phi\in \Op{Reg}(X,Z)$.
\end{lem}
Both the algebra of polynomially tame  maps  $\operatorname{PT}(X,Y)$ and the space of regular maps $\operatorname{Reg}(X,Y)$  depend not only on the  topologies of $X$ and $Y$ but also on the choice of the grading structures defined on them. Equivalent gradings (i.e.\ those with
identity map linearly tame of tameness degree $0$) provide the same algebras.

The above notions are motivated by the following simple examples:
\begin{eg}
Let $M$ be a closed manifold. To describe the regularity of a distribution  $u\in\maD'(M)$ we first note that any distribution 
provides a natural map between two Fr\'echet spaces, namely the space of smoothing operators $\smoothing (M)$ and the space of smooth functions $\smooth(M)$, by evaluation. More precisely, to any $u\in\maD'(M)$ we associate the map
\begin{align*}
\Theta_u: \smoothing(M) &\rightarrow\smooth(M)\\
\Theta_u(T):&=T(u)\quad  \\
\end{align*}

Fixing the grading on $\smooth(M)$ and on $\smoothing(M)$  as in Example  \ref{pseudo_closed},  we set $\operatorname{PT}(M):=\operatorname{PT}(\smoothing(M),$ $\smooth(M))$, and $\operatorname{Reg}(M):=\operatorname{Reg}(\smoothing(M),\smooth(M))$.
\begin{prop}\label{regularity} 
Let $\smoothing(M)$ and $\smooth(M)$ be graded as above. Then for any distribution $u\in \maD'(M)$ its image $\Theta_u:\smoothing(M)\rightarrow \smooth(M)$ is a polynomially tame map. In fact,
if $u\in H^k(M)$ then $\Theta_u\in\Op{PT}^{-k}(M)$.
\end{prop}
\begin{proof} Let $u\in H^k(M)$  be a distribution. Then:
\begin{align*}
\|\Theta_u(T)\|_n&=\|T(u)\|_n=\|(1+\Delta)^{\frac{n}{2}}T(u)\|_{L^2(M)}\\
&=\|(1+\Delta)^{\frac{n}{2}}T(1+\Delta)^{-\frac{k}{2}}\left( (1+\Delta)^{\frac{k}{2}}(u)\right)\|_{L^2(M)}\\
&\leq \|(1+\Delta)^{\frac{n}{2}}T(1+\Delta)^{-\frac{k}{2}}\|_{HS}\|(1+\Delta)^{\frac{k}{2}}(u)\|_{L^2(M)}\\
&\le C\|T\|_{n-k}\,,
\end{align*}
whenever $n\ge \max(k/2,2k)=: b$. Thus the tameness estimate is satisfied with this $b$ and $r= -k$.
\end{proof}
An important further extension of this result is the following characterization of 
regular distributions:
\begin{theo}
\label{MO}
The only distributions that give rise to regular maps are smooth functions:
\[\maD'(M)\cap \operatorname{Reg}(M)=\smooth(M).\]
\end{theo}
\begin{proof}
In view of Proposition \ref{regularity}, $\smooth(M)\subseteq \operatorname{Reg}(M)$. Thus 
it remains to show that given a nonsmooth distribution $u\in\maD'(M)$, the  map $\Theta_u$ is not contained in $\operatorname{Reg}(M)$. 

We first observe that there exists some $s\in \RR$ such that $v=(1+\Delta)^su\in L^2(M)$ and check that, if $\Theta_u\in \operatorname{Reg}(M)$, then so is $\Theta_v$. Thus without loss of generality
we may assume that $u\in L^2(M)$.  Let $\{\phi_i\}$ be an orthonormal basis of eigenvectors of $\Delta$ (with eigenvalues $\lambda_i\leq \lambda_{i+1}$ nondecreasing). Then 
\[u=\sum_na_n\phi_n\qquad (a_n)\in\ell^2\quad (a_n)\not\in \Sch(\NN).\]
We consider the action of $\Theta_u$ on smoothing operators which are represented by a diagonal matrix with respect to the basis $\{\phi_i\otimes \phi_j\}$, i.e., 
$$
T=\sum_n k_n\phi_n\otimes \phi_n
$$ 
with $(k_n)\in \Sch(\NN)$.
 
It then follows that $\|T\|_{p,q}^2=\sum_n |k_n|^2(1+\lambda_n)^{p+q}$,  while
$$
 \|\Theta_u(T)\|_k^2=\sum_n |a_nk_n|^2(1+\lambda_n)^{{k}}.
$$

Suppose now that $\Theta_u\in\Op{Reg}(M)$. By Remark \ref{lintameremark} this means that for any degree $r$ of
regularity there exists some $b_r$ in $\NN$ such that for each $k\ge b_r + |r|$
there exists some $C>0$ such that for each $(k_n) \in \Sch(\NN)$ we have
$$
\sum_n |a_n k_n|^2(1+\lambda_n)^k \le C \sum_n |k_n|^2(1+\lambda_n)^{k+r}.
$$
Noting that, by Weyl's estimates, 
the $\lambda_n$ asymptotically grow polynomially we
may rescale $(k_n)$ by $(1+\lambda_n)^{k/2}$ in this estimate to obtain:

$\forall r\in \ZZ$ $\exists C_r>0$ $\forall (k_n) \in \Sch(\NN)$:
$$
\sum_n |a_n|^2|k_n|^2 \le C_r \sum_n |k_n|^2(1+\lambda_n)^{r}.
$$
Inserting $k_n = \delta_{mn}$ this implies $(a_n) \in \Sch(\NN)$, contradicting
our assumption.

\end{proof}
\end{eg}

\begin{eg}\label{orderm}
Let $P$ be a pseudodifferential operator of order $m$. We first observe that right multiplication  by $P$  on $\smoothing(M)$, i.e., the map  $T\mapsto TP$  is polynomially tame of tameness degree $m$. To see this one notes that the operator $(1+\Delta)^{\frac{m}{2}}$ generates   $\Psi^m(M)$  as a  left-module (and also as a right-module) over $\Psi^0(M)$.  Therefore we may write
\[P=P_0(1+\Delta)^{\frac{m}{2}},\qquad P_0 \in\Psi^0(M).\]
 By the same token,  for $k$ an integer multiple of $\frac{1}{2}$  we find an order $0$ operator $T_k$ such that
\[[P_0,(1+\Delta)^k]=(1+\Delta)^{k-1/2}T_k.\]

\noindent Putting this together we have 
\begin{align*}
\|TP\|_{p,q}&=\|(1+\Delta)^{\frac{p}{2}}TP(1+\Delta)^{\frac{q}{2}}\|_{HS}=\|(1+\Delta)^{\frac{p}{2}}TP_0(1+\Delta)^{\frac{q+m}{2}}\|_{HS}\\
&\leq \|(1+\Delta)^{\frac{p}{2}}T(1+\Delta)^{\frac{q+m}{2}}P_0\|_{HS}+\\
&\qquad\qquad\|(1+\Delta)^{\frac{p}{2}}T[P_0,(1+\Delta)^{\frac{q+m}{2}}]\|_{HS}\\
&\leq \|P_0\|\|(1+\Delta)^{\frac{p}{2}}T(1+\Delta)^{\frac{q+m}{2}}\|_{HS}+\\
&\qquad\qquad\|T_{\frac{q+m}{2}}\|\|(1+\Delta)^{\frac{p}{2}}T(1+\Delta)^{\frac{q+m-1}{2}}\|_{HS}\\
&\leq C\|T\|_{p+q+m}
\end{align*}
In particular, this implies that
\[\|TP\|_n\leq C\|T\|_{n+m}.\]
\end{eg}
The following is a regularity result similar to Theorem \ref{MO}: 
\begin{prop} 
With the notations introduced above, 
\begin{eqnarray}\label{MO2}
\Psi^{\infty}(M)\cap \Op{Reg}(\smoothing(M))=\smoothing(M).
\end{eqnarray}
\end{prop}
\begin{proof}
Consider $P\in \Psi^{\infty}(M)$ such that $P\in\Op{Reg}(\smoothing(M))$.  Then by Lemma \ref{composition}, given any distribution $u\in\maD'(M)$, it follows that $\Theta_u\circ P=\Theta_{Pu}\in\Op{Reg}(M)$. By Theorem \ref{MO} this implies that $Pu\in\smooth(M)$ for all distributions $u$ and hence $P\in\smoothing(M)$.
\end{proof}

Thus we  see that regularity on maps between graded  Fr\'echet spaces gives back usual notions of regularity in familiar cases.

\subsection{Singularity structures}

As we have seen above, to introduce a notion of regularity one only needs graded Fr\'echet spaces and a notion of tameness of maps between them. But to further associate  ``singularities'' we need,  in addition, a Fr\'echet module structure with respect to a filtered algebra $A$. Singularities are then defined as certain ideals in the algebra $B=A_0/A_{-1}$. By $\sigma:A_0\rightarrow B$
we denote the symbol homomorphism. 

The following result is immediate from the definitions:

\begin{prop} \label{rmaction}
  Let $X$ be a right  Fr\'echet $A$ module and $Y$  a left Fr\'echet $A$ module. Then for each element $a\in A_j$, the right module  action on $PT(X,Y)$ has the following effect on the regularity:
  \[\Op{PT}^r(X,Y)\ni\phi\mapsto \phi\cdot a\in \Op{PT}^{r+j}(X,Y).\]
\end{prop}

\begin{definition}\label{abstract_ss}
 Let $X$ be a right  Fr\'echet $A$ module and $Y$  a left Fr\'echet $A$ module. We shall refer to  the data $(A,X,Y)$ as a singularity structure.
 \end{definition}
 
 In the sequel we shall drop the word  Fr\'echet for the sake of readability, but all modules will  be assumed to satisfy Definition \ref{Fmodule}.
 \begin{definition}Let $(A,X,Y)$ be a singularity structure. Let $B=A_0/A_{-1}$ and let 
 $\sigma:A_0\rightarrow B$ denote the symbol homomorphism. 
  Given a map $\phi:X\rightarrow Y$ in $\Op{PT}(X,Y)$ we consider the subset $\WF^A(\phi)$ of the algebra $B$ defined as
\[
\WF^A(\phi):=\{\sigma(a)~|~a\in A_0 ~\textrm{and}~ \phi\cdot a\in \Op{Reg}(X,Y)\}.
\]
By Proposition \ref{rmaction}, $\WF^A(\phi)$ is a right-ideal in $B$.
\end{definition}
\begin{remark}\label{idealrem} 
In case $B$ is the commutative algebra $\smooth(M)$ over some smooth manifold $M$, we shall associate to any 
ideal $\mathcal I$ in $\smooth(M)$ its zero set, i.e., the intersection of all zero-sets of 
elements of $\mathcal I$.
We shall see below that this notion reproduces the usual notion of wavefront set and even singular support for distributions and provides a ``propagation of singularities'' result in the context of Weyl algebras, which  agrees with usual propagation of singularities for distributions under appropriate conditions.
\end{remark}

The remainder of this section provides the motivation for the above definitions by examining classical examples of singularity structures.

\subsubsection{Singular support of a distribution}
\begin{eg}\label{singsup1} We start with a simple example where the algebra $A$ is $\smooth(M)$ with the trivial filtration $A_0=A$ and $A_{-1}=\{0\}$. A function $f\in A$ is acting on the right on $T\in\smoothing(M)$ by $T\mapsto TM_f$, where $M_f$ is the multiplication operator by $f$. Thus we consider the singularity structure $A=\smooth(M)=Y$ and $X=\smoothing(M)$. 
The following result shows that $\WF^A$ recovers the singular support of a distribution.
\end{eg}

\begin{lem}
For a distribution $u\in \maD'(M)$, the zero set of $\WF^{\smooth(M)}(\Theta_u)$
is $\operatorname{singsupp}u$.
\end{lem}
\begin{proof}
We have
$$
 \WF^{\smooth(M)}(\Theta_u) = \{ f \mid f \in \smooth(M),\, \Theta_u\cdot f \in 
 \Op{Reg}(M)\}
$$
and since $\Theta_u \cdot f = \Theta_{fu}$, Theorem \ref{MO} implies
$$
\WF^{\smooth(M)}(\Theta_u) = \{ f \mid f \in \smooth(M),\, f u \in \smooth(M)\}
$$
Since $\bigcap \{f^{-1}(0)\mid f\in \smooth(M),\ f\cdot u \in \smooth(M)\}=\operatorname{singsupp}(u)$, 
the result follows (cf.\ Remark \ref{idealrem}).
\end{proof}
\begin{remark}\label{singsupprem} 
We point out that, in general, the fact that a smooth function $f$ vanishes on the singular
support of a distribution $u$ does not imply that $f\cdot u$ is smooth. 
As a simple example take $M=[0,2\pi]$ the one-dimensional torus, $u:=\theta \mapsto \cos(\theta/2)$ 
and $f:=\theta\mapsto \theta$.
Therefore, $\WF^{\smooth(M)}(\Theta_u)$ contains in fact more information than $\operatorname{singsupp}u$.
\end{remark}

\subsection{Wave-Front set of a distribution}
Let us consider the singularity structure defined by the triple $A=\Psi^{\infty}(M)$, the algebra of classical pseudo-differential operators on a manifold $M$ and  (as in the last example) $X=
\smoothing(M)$ and $Y=\smooth(M)$ with the appropriate natural right and left  module structure from $A$. On $A$ we use the usual filtration by the order of the operators. The symbol 
homomorphism then indeed assigns its symbol to any order $0$ pseudodifferential operator. 
We are going to show that $\WF^A$ in this case is 
the usual wavefront set of a distribution.

Given a pseudodifferential operator $P$ we write $\sigma(P)$ for the principal symbol of $P$.  Moreover, we denote as usual the characteristic set of $P$ by $\Op{Char}(P)=\sigma(P)^{-1}(\{0\})
\subseteq T^*M$. 

\begin{prop}
Let $u\in \maD'(M)$  be a distribution and $A,X,Y$ be chosen as above. Then the 
zero set of $\WF^{\Psi^{\infty}(M)}(\Theta_u)$ in 
$T^*M$ is $\WF(u)$.
\end{prop}
\begin{proof}
Since $\Theta_u\cdot P=\Theta_{Pu}$ it follows from Theorem \ref{MO} that  
\[\Theta_uP\in\operatorname{Reg}(M)\Leftrightarrow\Theta_{Pu}\in \operatorname{Reg}(M)\Leftrightarrow Pu\in\smooth(M).
\]
Thus $\WF^{\Psi^{\infty}(M)}(\Theta_u) = \{\sigma(P) \mid P\in \Psi^0(M),\  Pu\in\smooth(M)\}$ and the corresponding zero set $T^*M$ is
$$
\bigcap_{\Theta_uP\in\Op{Reg}(M)} \Op{Char}(P)=\bigcap_{Pu\in\smooth(M)} \Op{Char}(P)=\WF(u).
$$
\end{proof}
Similar to Remark \ref{singsupprem}, also in the present situation it follows that $\WF^{\Psi^{\infty}(M)}(\Theta_u)$ in fact contains more information than $\WF(u)$.
\subsection{Microlocal ellipticity}
Let $A=\Psi^{\infty}(M)$ and let $X=Y=\smoothing(M)$.  As discussed in Example \ref{orderm}, right multiplication by a pseudo-differential operator $P$ of order $m$ is in $\Op{PT}^{-m}(\smoothing(M))$.
Thus from \eqref{MO2} it follows that the zero set of $\WF^{\Psi^{\infty}(M)}(P)$ in $T^*M$ is given by
$$
\bigcap_{PQ\in\Op{Reg}(\smoothing(M))} \Op{Char}(Q)\\
=\bigcap_{PQ\in\smoothing(M)} \Op{Char}(Q).
$$
These are precisely the directions in which $P$ is  microlocally elliptic.

\subsection{Spectral triples}\label{spectral_triples} One may naturally associate to a $p$-summable regular spectral triple $(A,H,D)$ a singularity structure, namely $(\maA_{(A,H,D)},\smoothing(A,H,D),\maH^{\infty})$. We shall denote this singularity structure by $\textrm{SS}(A,H,D)$.



\subsection{Abstract distributions}
We are now ready to formulate the notion of distributions for the more abstract setup. This generalization is motivated by the following observation.
\begin{prop}
On a closed manifold $M$ a polynomially tame map $\phi\in\mathrm{PT}(M)$ is defined by a distribution, that is $\phi=\Theta_u$ for some $u\in\maD'(M)$, if and only if
\[\phi:\smoothing(M)\rightarrow \smooth(M),\]
is a left $\Psi^{\infty}(M)$-module map.
\end{prop}
\begin{proof}\label{distributions}
It is clear that for any distribution $u$ the map $\Theta_u$ is a left $\Psi^{\infty}(M)$-module map, so we need only to prove the converse.
Let $\phi:\smoothing(M)\rightarrow \smooth(M)$ be a left $\Psi^{\infty}(M)$-module map.
 First let us assume  that  for a given approximate unit $T_{\ep}$ in the algebra $\smoothing(M)$, such that $\underset{\ep\rightarrow 0}{\textrm{lim}}~T_{\ep}T=\underset{\ep\rightarrow 0}{\textrm{lim}}~TT_{\ep}=T$ for all $T\in \smoothing(M)$, the limit $\underset{\ep\rightarrow 0}{\textrm{lim}}~\phi(T_{\ep})$ exists in distributions and is equal to $v$. Since by assumption 
 $\phi$ is a continuous module morphism, we have
 \begin{align*}
 \phi(T)&=\underset{\ep\rightarrow 0}{\textrm{lim}}~\phi(TT_{\ep})\\
 &=\underset{\ep\rightarrow 0}{\textrm{lim}}~T\phi(T_{\ep})=Tv = \Theta_v(T).
 \end{align*}
 Thus to complete the proof it suffices to show that the limit in the above assumption exists. Again since $\maD'(M)$ is a Montel space, it is enough to check the limit by evaluation on test-functions (smooth densities on $M$). We fix a Riemannian density to trivialize the density bundle and observe that given any function $f\in\smooth(M)$ there is a smoothing operator  $T$ such that $T(1)=f$\footnote{ If $\rho$  is a density that integrates to  $1$ then $T$ can be given by the kernel $ker(T)=f(x)\otimes \rho(y)$.}, and therefore
 \begin{align*}
 \underset{\ep\rightarrow 0}{\textrm{lim}}~\ip{\phi(T_{\ep}),f}&=\underset{\ep\rightarrow 0}{\textrm{lim}}~\ip{\phi(T_{\ep}),T(1)}\\
 &=\ip{\phi(T^*),1}
 \end{align*}
\end{proof}
 This Proposition inspires the following definition.
 \begin{definition}
 Let $X$ be a graded Fr\'echet bimodule over the filtered algebra $A$ and let $Y$ be a graded Fr\'echet module  over $A$. Then a distribution over the singularity structure $(A,X,Y)$ is a  polynomially tame left $A$-module map $\phi:X\rightarrow  Y$. The space of all distributions will be denoted by $D^A(X,Y)$.
 \end{definition}
 We shall continue to denote by $\maD'(M)$ the space of distributions over $(\Psi^{\infty}(M),\smoothing(M),\smooth(M))$.
 
\begin{eg}
The proof of Proposition \ref{distributions} can also be applied to the  singularity structure associated to Heisenberg calculus over a contact (or a Heisenberg) manifold $(M,H)$,  so that
the space of distributions over $(\Psi^{\infty}_H(M),\smoothing(M),\smooth(M))$ is again $\maD'(M)$. However, this prescribes a different notion of regularity and singularity as compared to distributions.
\end{eg} 

\begin{eg}\label{singsup2}
Let $\theta\in (0,1]$ be an irrational number. The rotation $C^*$ algebra $A_{\theta}$ is the universal algebra generated by two unitaries $U_1$ and $U_2$  satisfying the relation
\[U_2U_1=\lambda U_1U_2\quad \lambda=e^{2\pi i\theta}.\]
Recall that the noncommutative torus $\maA_{\theta}$ is the pre-$C^*$ algebra in $A_{\theta}$  of elements of the form 
\[\maA_{\theta}:=\{\sum a_{rs}U_1^rU_2^s~|~~a_{rs}\in\Sch(\ZZ^2)\}.\]
We consider an action of $\maA_{\theta}$ on the space $\maE=\Sch(\RR)$ of Schwartz functions on $\RR$ given by
\begin{align*}
\xi\cdot U_1(s):&=\xi(s+\theta) \\ 
\xi\cdot U_2(s):&=e^{2\pi i s}\xi(s) 
 \end{align*}
($\xi\in \Sch(\RR),\ s\in \RR$). 
 The smooth structure on $\maA_{\theta}$ corresponds to a natural action of the torus $\mathcal{T}^2$ on $A_{\theta}$. This structure is recovered from the  two canonical derivations on $\maA_{\theta}$ given on the generators by 
 \[\delta_j(U_k):=\delta_{jk}2\pi i U_k\quad (j,k=1,2).\]
These derivations lift to connections on the module  $\maE$, given by:
\[(\nabla_1\xi)(s):=-2\pi i \frac{s}{\theta}\xi(s)~~~\textrm{and}~~(\nabla_2\xi)(s):=\frac{d\xi}{ds}(s).\]
In addition, $\mathrm{End}_{\maA_{\theta}}(\maE)$, the space of  $\maA_{\theta}$ linear endomorphisms of $\maE$, can be naturally identified with $\maA_{\theta'}$ where  $\theta'=\frac{1}{\theta}$. More explicitly the generators are given by:
\[V_1\xi(s):=\xi(s+1)\qquad V_2\xi(s):=e^{-\frac{2\pi is}{\theta}}\xi(s).\]

 Thus, from the point of view of  the $\maA_{\theta}$ action, the following operators can be considered as order $\leq n$  differential operators on $\maE$, 
 \[D:=\sum_{\alpha+\beta\leq n}C_{\alpha\beta}\nabla_1^{\alpha}\nabla_2^{\beta}\quad C_{\alpha\beta}\in \maA_{\theta'}=\mathrm{End}_{\maA_\theta}(\maE).\]
 
 The principal symbol of the operator $D$ above is a function on the circle $S^1$  with values in  $\maA_{\theta'}$, given by:
 \[\sigma(D)(t):=\sum_{\alpha+\beta=n}C_{\alpha\beta}(\cos{t})^{\alpha}(\sin{t})^{\beta}.\]
In this setting, an operator is considered elliptic iff $\sigma(D)$ is invertible in $\smooth(S^1,\maA_{\theta'})$. We shall denote by $\maD_{\theta}$ the  above filtered algebra of differential operators on $\maE$.

We can now realize a noncommutative version of singular support of a Schwartz distribution analogous to Example \ref{singsup1}. First observe that the  algebra $\smoothing_{\Sch}(\RR)$ of operators on $\maE=\Sch(\RR)$ with kernels in $\Sch(\RR\times \RR)$ is also a module over the algebra of differential operators $\maD_{\theta}$.
 Further, fixing a Sobolev grading on $\maE$  with the positive operator $H:=-(\nabla_1^2+\nabla_2^2)$  and the corresponding tensor product grading on $\smoothing_{\maE}(\RR)$, both spaces become Fr\'echet modules over $\maD_{\theta}$. We consider the algebra $A =\textrm{End}_{\maA_{\theta}}(\maE) $ with  trivial grading $A_0:=A$ and $A_{-1}:=\{0\}$. It is clear that any $u\in\Sch'(\RR)$ defines an element $\Theta_u\in  D^A(\smoothing_{\maE}(\RR),\maE)$.
 
We now obtain, for instance, the wavefront set of the delta distribution 
 \[\mathrm{WF}^A(\delta_0):=\{a=\sum_nf_nV_1^n~|~~a\in \maA_{\theta'},\ f_n(n)=0.\}\]
 This singular support  is expected to propagate naturally under the Hamiltonian flow corresponding to elements of $\maD_{\theta}$.
 \end{eg}
\begin{eg}
The action of $\mathcal{T}^2$ on $A^{\theta}$ introduced in the previous example gives rise to a natural algebra of pseudodifferential operators on $\maA_{\theta}$ (\cite{NCG}).  For simplicity, we shall consider the lift of this action on $\RR^2$ and denote it by $\alpha$:
\[\alpha_x(U^i_1U^j_2):=e^{2\pi x\cdot (i,j)}U^i_1U^j_2\qquad (x\in \RR^2).\] 

The symbols of order $m$ in this calculus are maps $\rho:\RR^2\rightarrow \maA_{\theta}$ satisfying the following two conditions:
\begin{enumerate}
 \item  For all $i,j\in \NN$ and every multi-index  $\beta$ there exists $C_{ij\beta}>0$ such that
 \[\|\delta_1^i\delta_2^j(\partial^{\beta}\rho(\xi)) \|\leq C_{ij\beta}(1+|\xi|)^{m-|\beta|}\qquad (\xi \in \RR^2).\]
\item There exists  $\sigma\in \smooth(\RR^2\setminus{0})$ such that
\[\lim_{\lambda\rightarrow \infty}\lambda^{-m}\rho(\lambda\xi)=\sigma(\xi).\]
 \end{enumerate}
 By definition, a symbol $\rho $ is of oder  $-\infty$ if for its Fourier transform we have $\hat{\rho}\in \Sch(\RR^2,\maA_{\theta})$.

The operator on $\maA_{\theta}$ corresponding to a symbol $\rho$ is given by
\[P_{\rho}(a):=\int_{\RR^2}\int_{\RR^2}e^{-ix\cdot \xi}\rho(\xi)\alpha_x(a)dxd\xi.\]

 We shall denote the algebra of pseudodifferential operators of all orders by $\Psi^{\infty}(\maA_{\theta})$. 
With the seminorms provided  by $a\rightarrow \sum_{i+j=n}\|\delta^i_1\delta^j_2a\|$, $\maA_{\theta}$ turns into a Fr\'echet module over $\Psi^{\infty}(\maA_{\theta})$. 
As mentioned above, the ideal of smoothing operators can be identified with $\Sch(\RR^2,\maA_{\theta}) \cong \Sch(\RR^2)\hat\otimes \maA_{\theta}$. This space acquires a tensor product grading with the grading on $\Sch(\RR)$ as in Example \ref{singsup2}.
 Thereby, $(\Psi^{\infty}(\maA_{\theta}),\smoothing(\maA_{\theta}),\maA_{\theta})$ is a singularity structure.
 \end{eg}
A strong motivation for considering singularity structures and their distributions along with their regularity and singularity features lies in the natural  category of morphisms associated to them, to which we turn in the following section. 
\section{Some functoriality considerations} 
We wish to develop a category whose objects are singularity structures $(A,X,Y)$  as defined in \ref{abstract_ss} and suitable morphisms between them. To include the situation of the abstract distributions described in the previous section, we shall for the rest of the paper assume that $X$ is a Fr\'echet bimodule over $A$. The first choice of such morphisms is rather straightforward and we describe them below.
\begin{definition}\label{direct_morph}
Given two singularity structures $(A_j,X_j,Y_j)~j=1,2$,
a direct morphism $\alpha:(A_1,X_1,Y_1)\rightarrow (A_2,X_2,Y_2)$  is a triple of maps, consisting of
\begin{enumerate}
\item A morphism of filtered algebras $\alpha^a:A_1\rightarrow A_2$. This provides $X_2$ and $Y_2$  with (pullback) $A_1$ module structures.
\item  A tame linear map $\alpha^r:X_1\rightarrow X_2$ which is an $A_1$ bimodule map.
\item A tame linear map $\alpha^l:Y_2\rightarrow Y_1$ which is also a left-$A_1$ module map.
\end{enumerate}
We call the direct morphism $\alpha$ \textit{monic} if $\alpha^a$ as well as the map induced by it on the associated graded algebras is injective, the map $\alpha^r$ is injective, and $\alpha^l$ is surjective.  In addition, for monic direct morphisms we require that the maps $\alpha^l$ and $\alpha^r$ have bounded order of tameness.
\end{definition}
The above definition guarantees that any direct morphism $\alpha:(A_1,X_1,Y_1)\rightarrow (A_2,X_2,Y_2)$ induces a morphisms $\alpha^*: PT(X_2,Y_2)\rightarrow PT(X_1,Y_1)$ simply by composition
\begin{center}
$\xymatrix{
X_2\ar[r]^{\phi}&Y_2\ar[d]^{\alpha^r}\\
X_1\ar[u]^{\alpha^l}\ar@{-->}[r]^{\alpha^*\phi}&Y_1
}$
\end{center}
This has the further property that it preserves the $\Op{Reg}$-maps as well as distributions,
$\alpha^*:D^{A_2}(X_2,Y_2)$ $\rightarrow$ $D^{A_1}(X_1,Y_1)$. Then the traditional functoriality of  wavefront sets follows from the condition that $\alpha^l$ and $\alpha^r$ have bounded order of tameness, and in this case
\[\textrm{WF}^{A_1}(\alpha^*u) \subseteq \alpha_0^*(\textrm{WF}^{A_2}(u))\qquad u\in D^{A_2}(X_2,Y_2),\]
where $\alpha_0:\textrm{Gr}(A_1)_0\rightarrow \textrm{Gr}(A_2)_0$ is the map on the $0$-component of the associated graded algebras.

Let us consider a rather simple example of functoriality 
properties of singularity structures  relevant in classical analysis.
 Let $(A,X,Y)$  be a singularity structure and  $\rho:A'\rightarrow A$  a morphism of filtered algebras. By pulling back modules under $\rho$ we obtain a new singularity structure $(A',\rho^*X,\rho^*Y)$, where we shall suppress $\rho$ from the notation and write $(A,X,Y)$  instead for simplicity.
 Now $\rho$ induces a map $\tilde{\rho}:B':=A'_0/A'_{-1}\rightarrow B=A_0/A_{-1}$. For any $\phi\in \Op{PT}(X,Y)$ we note immediately that
 \[\tilde{\rho}:\WF^{A'}(\rho^*\phi)\rightarrow \WF^A(\phi).\]
 
 The inclusion $\smooth(M)\hookrightarrow  \Psi^{\infty}(M)$ as multiplication operators for instance provides the well known fact that the wavefront set of a distribution projects to its singular support. 
 
 \subsection{Correspondence morphism}
 Although the direct morphism defined in the previous section gives a nice functorial way of assigning wavefronts and regularity to abstract distributions, it turns out to be a rather restrictive concept as not all useful morphisms satisfy all the constraints of Definition \ref{direct_morph}. We shall here briefly describe a more general notion of correspondence and provide some simple examples. This should be thought of as a sort of ``Morita equivalence'' in the category $\mathfrak{Sing}$.

\begin{definition}
A correspondence between two singularity structures $(A_j,X_j,Y_j)~j=1,2$ is a span that is a third singularity structure $(A,X,Y)$ with two direct morphisms $\alpha_j:(A,X,Y)\rightarrow (A_j,X_j,Y_j)$
\begin{center}
$\xymatrix{
&(A,X,Y)\ar[dl]^{\alpha_1}\ar[dr]_{\alpha_2}&\\
(A_1,X_1,Y_1)&&(A_2,X_2,Y_2)
}$
\end{center}
We call a  correspondence \textit{true correspondence} provided the map $\alpha_1$ is monic in the sense of Definition \ref{direct_morph}
\end{definition}
 We shall abbreviate the data $(A,X,Y),\alpha_1,\alpha_2$ by $A$.
  
  Correspondences can be composed since the category $\mathfrak{Sing}_d$ with direct morphisms  has all pullbacks.  In particular, composition of two monic/true correspondences is again a true correspondence.  We shall refer to the category\footnote{Strictly speaking this is only a bicategory as compositions are only defined up to isomorphisms, we shall however not need this subtle distinction.} of singularity structures with true spans as morphisms as  $\mathfrak{Sing}$.
 
 \begin{definition} Let $A$  be a correspondence between two singularity structures $(A_i,X_i,Y_i) ~i=1,2$.
 An abstract distribution $u$ in $D^{A_2}(X_2,Y_2)$ is called transverse to $A$ provided there exists a  unique distribution $v\in D^{A_1}(X_1,Y_1)$ such that  $\alpha_1^*v=\alpha_2^*u$. In this case we set $A^*u:=v$.  In case the map $\alpha^a_2:\textrm{Gr}(A)_0\rightarrow \textrm{Gr}(A_2)_0$ is surjective then
 \[(\alpha^a_2)_*(\alpha^a_1)^*\textrm{WF}^{A_1}(v)\subseteq \textrm{WF}^{A_2}(u)\]
 \end{definition}
 The pull-back of distributions provides an example of a correspondence as follows:
 \begin{eg}
 Given a smooth surjective submersion $f:M\rightarrow N$ between two compact manifolds, let 
\[ \Psi^{\infty}_f(M):=\{P\in\Psi^{\infty}(M)~|P:f^*\smooth(N)\rightarrow f^*\smooth(N)\}.\]
Similarly we define $\smoothing_f(M):=\smoothing(M)\cap \Psi^{\infty}_f(M)$. We claim that the singularity structure $(\Psi^{\infty}_f(M),\smoothing_f(M),\smooth(M))$  with the morphisms given below describes a correspondence that represents the pull-back of distributions.

The obvious morphism $\alpha_1:(\Psi^{\infty}_f(M),\smoothing_f(M),\smooth(M))\rightarrow (\Psi^{\infty}(M),\smoothing(M),\smooth(M))$ is monic.  For $\alpha_2:(\Psi^{\infty}_f(M),\smoothing_f(M),\smooth(M))\rightarrow (\Psi^{\infty}(N),\smoothing(N),\smooth(N))$ we let  $\alpha_2^l=f^*$. Since $f^*$ is assumed to be injective, for  an operator $P\in\Psi^{\infty}_f(M)$ we can define    $\alpha_2^a(P)(\psi):={f^*}^{-1}P(f^*(\psi))$. 

The following lemma implies that this correspondence morphism indeed represents the pull-back of a distribution in the usual sense.
\begin{lem} Let $f:M\rightarrow N$ be a smooth surjective submersion. Let $u\in \maD'(N)$ be a distribution.  With the notations 
introduced above we have:
\[\alpha_1^* \Theta_{f^*u}=\alpha_2^*\Theta_u.\]
\end{lem}

 \end{eg}

 \section{Pseudo-differential operators on groupoids}\label{groupoid_calc}
In this section we shall construct a morphism of  singularity structures related to a Lie groupoid $\maG$ over the base manifold $M$ 
under the following two assumptions:
\begin{enumerate}
\item The  space of arrows $\maG^{(1)}$ is a compact manifold. This condition can be weakened but we shall keep this assumption for simplicity of exposition.
\item The anchor map on the Lie algebroid $\maA(\maG)$ (described below) is  surjective.
\end{enumerate}

 This morphism $\beta_{\maG}$ presented below is used to connect the propagation of singularities on $\maG$ to  the propagation on the base manifold $M$. We begin by recalling a few facts about the groupoid calculus.
 
Let $\maG$ be a differential groupoid over the space of units $M=\maG^{0}$.  We denote by $A(\maG)$ its Lie-algebroid and consider sections of $A(\maG)$ as right-translation invariant vector fields  on $\maG$ which are vertical with respect to the domain map $d$. Fix  a section of the vertical density bundle $\maD=|\Lambda|A(\maG)$  which provides a measure $\mu_x$ on each fiber $\maG_x:=d^{-1}x\quad (x\in M)$. As customary we denote by $r$ the range map and by $\maG^x:=r^{-1}x$ the fibers of $r$.

Given a vector bundle $E$ over  the space of units $M$, a pseudo-differential operator on $\maG$  of integer order $m$ is a differentiable\footnote{defined, for example, using smoothly varying symbols in local coordinates}  family of classical pseudo-differential operators $P:=(P_x)_{x\in M}$, where each $P_x$ is an element of $\Psi^m(\maG_x:r^*E)$ which is:
\begin{enumerate}
\item Uniformly supported: Let $k_x$ denote the distributional kernel of $P_x$ and define the support of $P$ as \[\textrm{supp}~P:=\overline{\bigcup_x \textrm{supp}~k_x},\]
Then $P$ is said to be uniformly supported if the set $\{gh^{-1}|(g,h)\in \textrm{supp}~P\}$ is compact in $\maG$.
\item Right invariant:
For any element $g\in\maG$ define a map $U_g:\smooth_c(\maG_{d(g)}:r^*E)\rightarrow \smooth_c(\maG_{r(g)}:r^*E)$ by
\[U_g(\phi)(g'):=\phi(g'g).\]
A family $P$ is right-invariant iff $U_gP_x=P_yU_g$  for all $g$  where $x=d(g)$ and $y=r(g)$.
\end{enumerate} 
 The space of all pseudo-differential operators of order $m$ is denoted by $\Psi^m(\maG:E)$.  The space of regularizing elements is denoted by $\Psi^{-\infty}(\maG:E):=\bigcap_m\Psi^m(\maG:E)$. It can  be identified with the convolution algebra $\smooth_c(\maG:\textrm{End}(E)\otimes d^*\maD)$. In case $E$ is the trivial line bundle we shall write $\Psi^{\infty}(\maG)$ for the corresponding algebra of pseudo-differential operators. The associated graded algebra can be identified with the space of homogeneous sections of $\smooth(A^*(\maG):\textrm{End}(E))$. An operator $P$ is then called elliptic if  $\sigma(P)$ is invertible away from the zero-section.

As mentioned before we will from now on assume that $\maG$ is compact.  To describe the singularity structure associated to $\maG$ we shall recall a few preliminaries. 
 Fix a choice of metric (nondegenerate, symmetric bilinear form) on $A(\maG)$  to determine a trivialization of $\maD$, and in addition  fix a Hermitian inner product on $E$. This gives rise a pre-Hilbert-* module over $\smooth(M)$ with field of Hilbert spaces given by $L^2(\maG_x,r^*E)$ for $x\in M$ and sections given by restriction of  continuous sections $\Gamma(\maG:r^*E)$ to each fiber $\maG_x$. We shall denote  the corresponding $C(M)$ $C^*$-modules by $\maE(M:E)$ and by $\maE$  when the vector bundle $E$ is trivial.
 
 The trivialization of $\maD$ provides an involution on $\Psi^{\infty}(\maG)$ . It follows that the
 elements of $\pdo{\maG}$ are $\smooth(M)$-linear (not necessarily bounded) regular operators on $\maE$.   Further, we fix a  form positive self-adjoint operator $D$ in  $\pdo{\maG}$.  Then $D$ is a regular operator on the $C^*$-module $\maE$  by the Woronowicz criterion (see, e.g., \cite{VS})  and hence defines a chain of Sobolev-modules $\maE^n$. 
 We set $\maE^{\infty}:=\bigcap_n \maE^n$ and use the norm of $\maE^n$ to provide the structure of a graded Fr\'echet space. It follows from the independence of the Sobolev chains  $\maE^n$ from the operator $D$ (see \cite{VS}) that $\maE^{\infty}$ is a Fr\'echet-module over $\pdo{\maG}$. 
 
 As $\maG$ is assumed to be compact,  the algebra $\smoothing(\maG)$ is a Fr\'echet module  over $\pdo{\maG}$ .  The appropriate grading on $\smoothing(\maG)$ will be described below.
  In this case one  obtains a $\Psi(C^*(\maG),C(M))$ unbounded module $(\maE, D)$.  The extension to equivariant vector bundles discussed below is analogous.  We are now ready to define $\textrm{SS}(\maG)$ as mentioned above.

 \begin{definition}\label{groupoid_SS}
 We denote by $\textrm{SS}(\maG)$  the singularity structure defined by $(\pdo{\maG},\smoothing(\maG),\maE^{\infty})$.
 \end{definition}

\subsection{Vector representation on the units and the associated singularity structure} 

The algebra $\pdo{\maG}$ consists of longitudinal operators on fibers $\maG_x$ of the domain map $d$. The vector representation assigns to each such operator a transversal operator, that is an operator on the base manifold $M$. This assignment gives a direct morphism from $\textrm{SS}(\maG)$ to the singularity structure on $M$. First we recall some basics about the vector representation (see \cite{NWX} for more details).

 Recall that the  trivialization of the density bundle $\maD$  chosen above also provides  a norm on $\Psi^{-\infty}(\maG)$ given by 
  \[\|P\|_1:=\underset{x\in M}{\sup}\left\{\int_{G_x}k_Pd\mu_x,\int_{G^x}k_pd\mu_x\right\}\]
  Any representation of $\Psi^{-\infty}(\maG)$ that  satisfies $\|\pi(P)\|\leq \|P\|_1$ extends uniquely to a bounded operator on $\Psi^{\infty}(\maG)$, where by bounded one means that $\Psi^0(\maG)$ acts by bounded operators. We shall only consider representations $\pi$ of $\smoothing(\maG)$ that are non-degenerate, that is $\pi(\smoothing(\maG))H=\maH^{\infty}$ is dense in $H$ and the above condition on the norms holds. For such representations any self-adjoint elliptic operator $Q$ of order $1$ gives rise to an essentially self-adjoint operator $\pi(Q)$. In the standard way one can define Sobolev space completions of $\maH^{\infty}$ using powers of $Q$ and such Sobolev spaces do not depend on the choice of the elliptic operator $Q$. Thus consider $\maH^{\infty}$ with  the graded Fr\'echet space structure provided by these Sobolev norms.
  
 There is an obvious action  of $\Psi^{\infty}(\maG:E)$ on  $\smooth(M:E)$, called the vector representation, obtained by setting $P(f)\circ r=P(f\circ r)$.  We shall denote this action by $\pi_0$. When the space of units is a compact manifold (possibly with corners) then $(\Psi^{\infty}(\maG:E),\Psi^{-\infty}(G:E),\smooth(M))$ provides an example of a singularity structure.  In case $E$ has the structure of an equivariant vector bundle, this gives rise to a representation of $\pdo{\maG}$  as described below  \cite{NWX,LN}.
  
 Let us consider a special case of vector bundle. An equivariant vector bundle $(V,\rho)$ on a groupoid $\maG$ is a vector bundle $V$ over the units $M$  provided with a bundle isomorphism $\rho: d^*V\rightarrow r^*V$ such that $\rho(gh)=\rho(g)\rho(h)$.
 Given an equivariant vector bundle $(V,\rho)$ there is an action of $\Psi^{\infty}(\maG)$ on $\smooth(M:V)$ defined  
  by identifying $\smooth_c(\maG)\otimes V_x\simeq \smooth(\maG:d^*V)$ and hence defining a map $T_{\rho}:\Psi^{\infty}(\maG)\rightarrow \Psi^{\infty}(\maG,V)$, given by
  \[T_{\rho}(P):=\rho(P\otimes 1)\rho^{-1}\]
  and thus defining an action on $\smooth(M:V)$ by $\pi_{\rho}:=\pi_0\circ T_{\rho}$. In fact, analogous to Definition \ref{groupoid_SS} we can define a singularity structure $\textrm{SS}(\maG,V):=(\pdo{\maG},\smoothing(\maG),\maE^{\infty}(\maG,V))$. All the results in this section generalize to the case of equivariant vector bundles, we shall however only formulate the case of a trivial vector bundle $V$ and refer to the map $\pi_0$  as $\pi$.

  As mentioned above we assume that the anchor map $q:A(\maG)\rightarrow TM$ is surjective. This provides a natural measure on $M$ such that the vector representation becomes a bounded $*$ representation. In addition, it provides a direct morphism from $\textrm{SS}(\maG)$ to the natural singularity structure on $M$.
  \begin{lem}\label{vect_rep}
  Let $\maG$ be a compact Lie groupoid such that the anchor map $q:A(\maG)\rightarrow TM$ is surjective. Then the following hold:
  \begin{enumerate}
\item  There is a measure on $M$ such that the vector representation gives rise to a bounded $*$ representation  of $\smoothing(\maG)$ on  $L^2(M,V)$. 
 \item There is a form-positive  elliptic self-adjoint operator in\ $\Psi^1(\maG)$  whose image  $\pi_{\rho}$ is an ordinary elliptic essentially self--adjoint operator of order $1$. 
 \end{enumerate}
  \end{lem}
 \begin{proof}
 Since we have fixed a metric on $A(\maG)$ it identifies the complement $ker(q)^{\perp}$ with $TM$  and hence provides a Riemannian metric on $M$. We shall denote the measure obtained from this metric by $\mu_M$. In addition, the surjectivity of the anchor map implies that for any $x\in M$ the range map $r:\maG_x\rightarrow M$ is a submersion. Therefore the Riemannian metric on $\maG_x$ induced from $A(\maG)$  splits into a tensor product of a density on fibers of $r$, that is along the submanifolds $\maG_x^y:=d^{-1}x\cap r^{-1}y$, and a normal density $\mu_M$. For a smooth function $\phi\in\smooth(M)$ integration along the fiber gives
 \begin{align*}
 \int_{\maG_x}r^*\phi\, d\mu_x&=\int_M\left(\int_{\maG^y_x}\phi(y)d\mu_x^y\right)d\mu_M(y)\\
 &=\int_M \phi(y)\textrm{vol}(\maG^y_x)\,d\mu_M(y).
 \end{align*}
 Therefore (on the connected component of $x\in M$) we choose the measure $\textrm{vol}(\maG_x^y)d\mu_M$ and it is clear that on the corresponding space $L^2(M,V)$ the representation of $\smoothing(\maG)$ is a bounded $*$ representation.
 
 To show that there is an elliptic operator in the image of the vector representation of $\pdo{\maG}$  one observes that  a $d$-vertical right-invariant vector field $X$ on $\maG$  can be considered as an element in $\Psi^1(\maG)$. If $s$ is the section of $A(\maG)$ associated to the right-invariant vector field $X$ then the vector representation  $\pi(X)$ equals the action of the vector field $q(s)$. Since the anchor map $q$ is assumed to be surjective, all vector fields and hence all differential operators on $M$ lie in the range of the vector representation $\pi$.  In fact these can be obtained from geometrical differential operators  on $\maG$.  (This assertion is equivalent to the fact that the geometrical differential operators on $\maG$, that is the space of  families of right-invariant $d$-vertical differential operators  on $\maG_x$, can be identified with the symmetric algebra of $\Gamma(M:A(\maG))$.)
In particular, the bilinear form on $A(\maG)$ gives a right invariant family of metrics on $\maG_x$ as $x$ varies in $M$ and the corresponding Laplace operators $\Delta_x$ form a family of elliptic self-adjoint operators on the module $\maE$ and $\pi(\Delta_x)$ is an elliptic differential operator. Thus we may choose the form-positive operator $D=\sqrt{\Delta_x+c}$.  It follows  that $\pi(D)$ is an elliptic pseudo-differential operator on $M$ as desired.
\end{proof}
 
 By applying results from \cite{LN} one can now  construct the desired direct morphism $\textrm{SS}(\maG)\rightarrow \textrm{SS}(M)$.
\begin{theo}
Under the same assumptions as in Lemma \ref{vect_rep}, the vector representation provides a direct morphism
$\beta_{\maG}:\textrm{SS}(\maG)\rightarrow \textrm{SS}(M)$ by
\[\pi:\pdo{\maG}\rightarrow \pdo{M},\quad r^*:\smooth(M)\rightarrow \maE^{\infty}. \]
\end{theo}
\begin{proof}
Our first goal is to show that the vector representation of an element $P\in\pdo{\maG}$ is in fact a classical pseudo-differential operator on $M$ of the same order, thereby providing a map of filtered algebras $\pi:\pdo{\maG}\rightarrow \pdo{M}$.
It is clear from Lemma \ref{vect_rep} that $H=L^2(M)$ is a bounded $*$ representation and by \cite{LN} the Sobolev spaces $H^s$ defined from a positive elliptic  operator  $D\in \Psi^1(\maG)$  over a bounded representation are independent of the choice of $D$. Furthermore, it can be shown that any $P\in\Psi^m(\maG)$ provides a bounded operator $\pi(P):H^s\rightarrow  H^{s-m}$. From \ref{vect_rep} it follows that  $\pi(D)$, being an elliptic pseudo-differential operator,  defines  the usual Sobolev spaces, that is  $H^s=H^s(M)$.    By a classical theorem of Kohn-Nirenberg \cite{KN} any continuous operator $L:\smooth(M)\rightarrow \smooth(M)$  whose iterative commutators with any given finite set of vector fields $Y_j$, $0\leq j\leq n$, has the property that $[Y_1,[\ldots [Y_n,L]]\ldots]$ extends to a continuous linear map $H^s(M)\rightarrow H^{s-m}(M)$ is a (not necessarily classical)  pseudo-differential operator of order $m$ on $M$. 

However, verifying that $\pi(P)$ is a classical pseudo-differential operator requires the theory of Fourier integral operators.  Briefly, on regularizing operators the vector representation $\pi:\smoothing(\maG)\rightarrow \smoothing(M)$ can be viewed as a map on smoothing kernels $\pi:\smooth(\maG:d^*\maD)\rightarrow \smooth(M\times M:\pi_2^*|\Lambda|(M))$  and can be identified with integration along fibers of the map $r\times d:\maG\rightarrow M\times M$. The integration along fibers is a Fourier integral operator with distributional kernel in $I^0(\maG\times M\times M,\mathrm{graph}(r\times d):\operatorname{End}(d^*\maD,d^*|\Lambda|(M))$  which extends  naturally to an operator $(r\times  d)_*:I^m(\maG,M:d^*\maD)\rightarrow I^m(M\times M,\Delta:\pi_2|\Lambda|(M))$. This extension, after identification with distributional kernels, is the vector representation on $\pdo{\maG}$.
 This shows that $\pi:\pdo{\maG}\rightarrow \pdo{M}$ is a morphism of filtered algebras. 

The defining relation $P\circ r=r\circ \pi(P)$ for the vector representation implies immediately that $r^*:\smooth(M)\rightarrow \maE^{\infty}$ is a module map over  $\pdo{\maG}$. The tameness of the map $r^*$ follows from  the following simple estimate.
\begin{align*}
\|r^*\phi\|_{\maE}&=\underset{x}{\textrm{sup}}\left(\int_{\maG_x}|r^*\phi|^2d\mu_x\right)^{\frac{1}{2}}\\
&= \underset{x}{\textrm{sup}}\left(\int_{M}|\phi|^2 \textrm{vol}(\maG_x^y)^2d\mu_M\right)^{\frac{1}{2}}\leq C\|\phi\|_{L^2(M)}.
\end{align*}

The  tameness of $\pi:\smoothing(\maG)\rightarrow \smoothing(M)$ will follow immediately from the definition of the  grading on $\smoothing(\maG)$.
\end{proof}

As a consequence of the above morphism $\beta_{\maG}$, a distribution on $M$ can be considered as a distribution on $\textrm{SS}(\maG)$. Thus if $u\in\maD'(M)$ then 
\[\beta_{\maG}^*\Theta_u=r^*(\pi(T)u)\qquad T\in\smoothing(\maG).\]
  By assumption, $q:A(\maG)\rightarrow TM$ is surjective and hence the dual map $q^*:T^*M\rightarrow A^*(\maG)$ is injective, which gives a map between the symbol spaces $\bar{q}:\smooth(A^*(\maG)\setminus{\{0\}})\rightarrow \smooth(T^*M\setminus{\{0\}})$.
  \begin{lem} \label{anchor}The action of $\pi$ on the associated graded algebras is given by $\bar{q}$.
  Thus $\WF(\beta_{\maG}^*\Theta_u)$ is the pull-back $\bar{q}^*(\WF(\Theta_u))$.
  \end{lem}
  \begin{proof}
  It suffices to show that $\bar{q}$  agrees with $\textrm{Gr}(\pi) $ on a dense subspace of the space of symbols. Observe that if $s$ is a section of $A(\maG)$ and $X_s$ is the right-invariant vector field in $\pdo{\maG}$ determined by it, then:
  \[ \sigma(\pi(X_s))=\sigma(q(s))=\bar{q}(\sigma(X_s)).\]
  This implies that $\bar{q}$ maps the symbol $\sigma(D)$  of a right-invariant differential operator $D$ on $\maG$ to the symbol of the operator $\pi(D)$.  Thus the map  $\bar{q}$ matches with  $\textrm{Gr}(\pi)$, the map  induced by $\pi $ on the associated graded algebra, on a dense subspace of polynomial functions.
  \end{proof}
\section{Propagation of singularities}

From now on  let us assume that the filtered algebra $A$  has a locally convex topology with respect to which the multiplication is at least separately continuous.

We shall consider an action of the group $\RR$ on a singularity structure $(A,X,Y)$ via direct morphisms, that is  $t\rightarrow \alpha_t:=(\alpha^a_t,\alpha^r_t,\alpha^l_t)$.  We shall simplify the notation by writing $t\cdot a, x\cdot t,t\cdot y$, respectively, for the action on the spaces $A$,  $X$, and $Y$.
A dynamics on a singularity structure $(A,X,Y)$ corresponds to  the action of $\RR$ on  $A$  being compatible with another action on $X$ in the following sense:
\begin{equation}\label{compatible}
(x\cdot t)a=\left(x(t\cdot a)\right)\cdot t.
\end{equation}
 for all $t\in \RR$ and $a\in A$, $x\in X$. A similar compatibility insures that the action on $Y$ provides a suitable $A$-module morphism. 
Clearly, actions of $\RR$  lead to a dynamics in  $\Op{PT}(X,Y)$  
 and we set 
\[\WF^A_t(\phi):=\WF^A(\phi\cdot  t) \qquad\phi\in\Op{PT}(X,Y).\]
In the classical case of $(\Psi^{\infty}(M),\smoothing(M),\smooth(M))$ such a dynamics is provided by solutions to  first order  hyperbolic  partial differential equations.  The following notion is motivated from this point of view.
 \begin{definition}
 We say that an action of $\RR$ on  a unital algebra is generated by an element $a\in A$ if for any element $p$
 \[\left.\frac{d^n}{dt^n}t\cdot p\right|_{t=0}=\delta_a^n(p), ~~~n>0.\]
 where $\delta_a(p)=[a,p]$
 \end{definition}
 Typically, given a self-adjoint elliptic operator $Q\in \Psi^1(M)$ on a closed manifold, an action of $\RR$ on $\pdo{M}$ is given by $t\cdot P=P_t=e^{itQ}Pe^{-itQ}$.  The standard Egorov theorem says that $P_t$ is  also a pseudo-differential operator  whose principal symbol corresponds to the symplectic Hamiltonian flow of the symbol of $P$ with Hamiltonian $\sigma(Q)$. Such an action naturally extends to the case of families of right-invariant operators on a groupoid.
 
  In general, to obtain such an $\RR$ action on a spectral triple $(A,H,D)$ by conjugation  
with $e^{it|D|}$, certain conditions  need to be imposed. Although the simple condition 
\[\underset{n}{\limsup} \frac{\textrm{Log}\|\delta^n(a)\|}{n}< \infty,\]
insures the convergence,
\[e^{it|D|}ae^{-it|D|}=\sum_{j}\frac{t^j\delta^j(a)}{j!},\]
it does not guarantee a priori that  the resulting operator is a pseudo-differential operator.

In the following section we shall work only with commutative examples where actions are generated by an element of $A$.
\subsection{The commutative case}
The classical theorems about the propagation of singularities  correspond to examples of 
singularity structures $(A,X,Y)$ such that the algebra $B=A_0/A_{-1}$ is a nice  commutative algebra and the generator $a$ of the  action 
is a derivation on $B$ through the commutator $A_0\ni a_0\rightarrow [a_0,a]$. A prototypical example is provided by  a  Weyl algebra $A$ in the sense of Guillemin \cite{Guillemin}.

Let  $C$ be a separable $C^*$-algebra. In what follows,  $\mathbf{B}(\maH)$ and $\mathbf{K}(\maH)$ denote the space of bounded and
compact $C$-linear operators on a Hilbert $C$-module $\maH$, respectively.  We shall suppress  mentioning the algebra $C$  in the notation.

\begin{definition}
 A singularity structure $(A,X,Y)$  is called  commutative (over the $C^*$-algebra $C$) if the algebra $A$ satisfies the following conditions:
 \begin{enumerate}[(a)]
  \item The algebra $A=\bigcup_{i\in\ZZ}A_i$  is a unital  *-algebra of (possibly unbounded) regular operators on a separable Hilbert $C$-module $\maH$.
  \item $A_0\subset \mathbf{B}(\maH)$ and $A_{-1}\subset \mathbf{K}(\maH)$.
  \item\label{commutator} The bracket  reduces the filtration order by $1$, that is 
  \[a\in A_i,\ b\in A_j~\Longrightarrow~[a,b]\in A_{i+j-1}.\]
  The latter two conditions together imply that $B=A_0/A_{-1}$ is a unital commutative subalgebra of the  Calkin algebra $\mathcal{Q}(\maH):=\mathbf{B}(\maH)/\mathbf{K}(\maH)$. The $C^*$-completion $\bar{B}$ of $B$ is a commutative  $C^*$ algebra and hence $\bar{B}=\mathcal{C}(N)$ for some compact Hausdorff space $N$. We shall further assume that
  \item 
The space $N$  is a compact manifold, $B=\smooth(N)$, and the commutator  with an element $a\in A_1$ defines a smooth vector field on $N$. More precisely,
  given $a\in A_1$ there exists a unique vector field $V_a \in\mathfrak{X}(N)$ such that:
  \[\sigma([a_0,a])=V_a(\sigma(a_0)).\]
  Note that it follows from \eqref{commutator} above that the expression on the  left hand side of the above equation is well-defined.
 \end{enumerate}
\end{definition}

\begin{theo}\label{propagation}
Let $(A,X,Y)$ be a commutative singularity structure and let $a\in A_1$ generate an action on   $(A,X,Y)$. Let $\mu_t$ be the flow generated by the vector field $V_a$. Then for any $\phi\in PT(X,Y)$, 
 \[\WF^A(\phi\cdot t)=\mu_t^*\WF^A(\phi).\]
\end{theo}
\begin{proof}
Set $\phi_t=\phi\cdot t$ for $\phi\in\Op{PT}(X,Y)$. 
 For any $\phi\in PT(X,Y)$, abbreviate $\WF^A(\phi)$ by $I_{\phi}$. Let $P\in A_0$, then it follows that:
 \begin{align*}
  \sigma(P)\in I_{\phi_t}&\Leftrightarrow \phi_t P\in \Op{Reg}(X,Y)\\
       &\Leftrightarrow \phi (t^{-1}\cdot P)\in \Op{Reg}(X,Y)\\
            &\Leftrightarrow \sigma(t^{-1}\cdot P)\in I_{\phi}
\end{align*}
Thus we conclude
\begin{align*}
 \frac{d}{dt}\sigma(t^{-1}\cdot P)&=-\sigma([a,t^{-1}P])\\
                                  &=V_a(\sigma(t^{-1}\cdot P))
\end{align*}
\end{proof}
Next we show that various propagation results can be explained in this framework.
\subsubsection{}
For the sake of  simplicity let $M$ be a closed compact oriented Riemannian manifold.
Let $D$ be an order $1$ self-adjoint pseudodifferential operator on $L^2(M)$. Let $u_0\in \maD'(M)$ be a distribution  and consider the Cauchy problem: 
\begin{align*}
 \frac{d}{dt}u(x,t)&=iDu(x,t)\\
 u(-,0)&=u_0
\end{align*}
The solution operator $e^{itD}$ defines an action 
on $\smoothing(M)$ by  
$\beta_t(T):=Te^{-itD}$, as well as an action on $\Psi^{\infty}(M)$  by $\alpha_t(P):=e^{itD}Pe^{-itD}$.
These two actions are certainly compatible in the sense of (\ref{compatible}). The singularity structure $(\Psi^{\infty}(M),\smoothing(M),\smooth(M))$ is commutative and hence Theorem \ref{propagation} applies to all $\phi\in \Op{PT}(M)$. Thus the wavefront set of $\phi$  propagates along the Hamiltonian flow of $\sigma(D)$. This is of course well known when $\phi=\Theta_u$ by \cite{Hormander}, as one clearly has $\Theta_u\cdot t=\Theta_{e^{itD}u}$, which is the solution to the Cauchy problem above.

If one considers maps in $\Op{PT}(M)$ as generalized functions extending the distributions, then as a consequence of Theorem \ref{propagation} propagation of singularities of generalized functions also follows. 
\subsection{}Given a groupoid $\maG$ satisfying the conditions in Section \ref{groupoid_calc}, we observe that  $SS(\maG)$ is a commutative  singularity structure over the algebra $C(M)$, where $M$ is the space of units. Given an elliptic order one  operator $D\in\Psi^1(\maG)$, the fiberwise Egorov theorem on $\maG_x$ gives an action on $P\in\pdo{\maG}$ by $P\mapsto e^{itD}Pe^{-itD}:= (e^{itD_x}P_xe^{-itD_x})_x$.  As in the previous section this defines an $\RR$-action on $\textrm{SS}(\maG)$.

\begin{theo}
Let $D$ be an elliptic order one operator in $\pdo{\maG}$. Consider the $\RR$-action  on $\textrm{SS}(\maG)$  generated by $D$, and correspondingly the action  on $\textrm{SS}(M)$ generated by $\pi(D)$.  The vector representation $\pi:\textrm{SS}(\maG)\rightarrow \textrm{SS}(M)$ is an equivariant morphism.
\end{theo}
\begin{proof} 
 We need to show that $\pi( e^{itD}Pe^{-itD})= e^{it\pi(D)}\pi(P)e^{-it\pi(D)}$. 

For $f\in \smooth(M)$  one can obtain $e^{it\pi(D)}\pi(P)e^{-it\pi(D)}f$ as the solution to the Cauchy problem
\begin{align*}
 \frac{d}{dt}u(x,t)&=i\pi(D)u(x,t)\\
 u(-,0)&=P(g)
\end{align*}
where $g$ is also the unique solution to some Cauchy problem, namely
\begin{align*}
 \frac{d}{dt}u(x,t)&=iDu(x,t)\\
 u(-,0)&=f
\end{align*}
  Pulling back with $r$ these equations and the initial data, one obtains the desired equality by appealing to the uniqueness of the solution to the mentioned Cauchy problem. 
 
\end{proof}

Thus, Lemma \ref{anchor} relates the dynamics of singularities under longitudinal propagation of the distributions $r^*\Theta_u$ (as seen in the symbol space $\smooth(A^*(\maG)\setminus{0})$) with that of  the distribution $u$ under 
the transverse action in vector representation via the anchor map.

 \subsection*{Acknowledgment} 
 Shantanu Dave would like to thank for the hospitality of the Institut des Hautes \'Etudes Scientifiques (I.H.\'E.S.), where a part of this work was finished.






\begin{thebibliography}{1}
 \bibitem{BJ} {\sc Baaj, S., Julg, P.}, {Th\'eorie bivariante de {K}asparov et op\'erateurs non
              born\'es dans les {$C^{\ast} $}-modules hilbertiens}, {\em C. R. Acad. Sci. Paris S\'er. I Math.}, Vol {296}, {1983}, no. {21}, {875--878}. 
  
  \bibitem{BG}{\sc Beals, R., Greiner, P.}, \emph{Calculus on {H}eisenberg manifolds}, {Annals of Mathematics Studies}, Vol.{119}, {Princeton University Press}, {1988}.
     

\bibitem{BGV} {\sc Berline, N., Getzler, E., Vergne, M.}, {\em Heat kernels and Dirac operators. Corrected reprint of the 1992 original}. Grundlehren Text Editions. Springer-Verlag, Berlin, 2004.

\bibitem{NCG} {\sc Connes, A.}, {\em Noncommutative geometry.} Academic Press, San Diego, CA, 1994. 

\bibitem{CCM} {\sc Connes, A., Consani, C., Marcolli, M.},  {Noncommutative geometry and motives: the thermodynamics of
              endomotives}, {\em Adv. Math.}, vol {214}, {2007}, no. {2}, {761--831}.
 
\bibitem{CM} {\sc Connes, A., Moscovici, H.}, The local index formula in noncommutative geometry. 
{\em Geom. Funct. Anal.} 5 (1995), no. 2, 174--243.


\bibitem{Guillemin}
{\sc Guillemin, V.},
\newblock A new proof of {W}eyl's formula on the asymptotic distribution of
  eigenvalues.
\newblock {\em Adv. in Math.}, 55(2):131--160, 1985.

\bibitem{GS} {\sc Guillemin, V., Sternberg, S.}, \emph{Geometric asymptotics},
{Mathematical Surveys, no. 14}, 1977.
 
\bibitem{Hamilton}
{ {\sc Hamilton, R.}}, {The inverse function theorem of Nash and Moser}, {\em Bull. Amer. Math. Soc. (N.S.)}, vol
{7} 1982, 1, 65--222.

\bibitem{Higson} {\sc Higson, N.}, The residue index theorem of Connes and Moscovici. Surveys in noncommutative geometry, 71--126, {\em Clay Math. Proc.}, 6, Amer. Math. Soc., Providence, RI, 2006.

\bibitem{Hormander}{\sc  {H\"ormander L.}}, \emph{The analysis of linear
partial differential operators}, Classics in Mathematics,
Springer-Verlag, Berlin, 2003.

\bibitem{KN}
{\sc Kohn, J., Nirenberg, L.}, {An algebra of pseudo-differential operators},
  {\em Comm. Pure Appl. Math.}  (1965), 260--305.
    


\bibitem{Kucerovsky}{\sc Kucerovsky, D.}, {The {$KK$}-product of unbounded modules}, {\em $K$-Theory},
    Vol. {11}, {1997}, no. {1}, {17--34}.
    
    \bibitem {LN}{\sc Lauter, R., Nistor, V.},
     \emph {Analysis of geometric operators on open manifolds: a groupoid
              approach},
  {Quantization of singular symplectic quotients}, {Progr. Math.},
    Vol {198}, {181--229}, {2001}. 
   
  
  \bibitem{Melrose}{\sc Melrose, R. B., Sj{\"o}strand, J.}, {Singularities of boundary value problems. {I}},
 {\em Comm. Pure Appl. Math.},
   Vol {31}, {1978}, no. {5}, {593--617}.
   
   \bibitem{MVW} {\sc Melrose, R., Vasy, A., Wunsch, J.}, {Propagation of singularities for the wave equation on edge manifolds}, {\em Duke Math. J.},
  Vol. {144}, {2008}, no. {1}, {109--193}.
    
  \bibitem{MW}{\sc Melrose, R. B., Wunsch, J.}, {Singularities and the wave equation on conic spaces}, {\em Geometric analysis and applications ({C}anberra, 2000)}, {Proc. Centre Math. Appl. Austral. Nat. Univ.},
    Vol. {39}, {170--182}, {2001}.
     
     
    \bibitem{Mesland}{\sc Mesland, B.}, {Unbounded bivariant $K$-theory and correspondences in noncommutative geometry},
    {\it preprint~}arXiv:0904.4383.
     
     \bibitem{BM}{\sc Monthubert, B.}, {Groupoids and pseudodifferential calculus on manifolds with corners}, {\em J. Funct. Anal.},
  Vol. {199}, {2003}, no.{1}, {243--286}.
     
    \bibitem{NWX}{\sc Nistor, V., Weinstein, A., Xu, P.},
      {Pseudodifferential operators on differential groupoids},
   {\em Pacific J. Math.}, {189},{1999},
    no. {1}, {117--152}.
      

\bibitem{Erik1} {\sc van Erp, E.},  {The {A}tiyah-{S}inger index formula for subelliptic operators
              on contact manifolds. {P}art {I}}, {\em Ann. of Math. (2)},
    Vol. {171}, {2010}, no. {3}, {1647--1681}.
    
    \bibitem{VS} {\sc Vassout, S}, {Unbounded pseudodifferential calculus on Lie groupoids},
   {\em J. Funct. Anal.}, Vol. {236},{2006},
    no. {1}, {161--200}.
    
  \bibitem{Vasy} {\sc Vasy, A.}, {Propagation of singularities for the wave equation on
              manifolds with corners}, {\em Ann. of Math. (2)},
    Vol. {168}, {2008}, no. {3}, {749--812}.
    
  


\end{thebibliography}
\end{document}